\documentclass[11pt]{article}
\usepackage{epigamath}

\usepackage[notext]{kpfonts}
\usepackage{baskervald}

\setpapertype{A4}

\usepackage[english]{babel}

\usepackage[dvips]{graphicx}     

\title{$p$-adic lattices are not K\"ahler groups}
\titlemark{$p$-adic lattices are not K\"ahler groups}
\author{\vspace{0cm} Bruno Klingler}
\authoraddresses{
\authordata{Bruno Klingler}{\firstname{Bruno} \lastname{Klingler}\\
\institution{Humboldt-Universit\"at zu Berlin, Germany}\\
\email{bruno.klingler@hu-berlin.de}}
}
\authormark{B. Klingler}
\date{\vspace{-5ex}} 
\journal{\'Epijournal de G\'eom\'etrie Alg\'ebrique} 
\acceptation{Received by the Editors on September 21, 2018, and in final form
on January 22, 2019. \\ Accepted on March 21, 2019.}

\newcommand{\articlereference}{Volume 3 (2019), Article Nr. 6}

\acknowledgement{B.K.'s research is supported by an Einstein Foundation's
  professorship}

\usepackage[all]{xy}

\allowdisplaybreaks

\newcounter{subsectionnum}
\numberwithin{subsectionnum}{numsection}
\renewcommand{\thesubsectionnum}{\arabic{numsection}.\Alph{subsectionnum}}
\newlength{\lengthtitle}
\newcommand{\newsubsection}[1]{
\settowidth{\lengthtitle}{#1}
\ifnum\lengthtitle=0\paragraph{\bfseries \thesubsectionnum.}\else\paragraph{\bfseries \thesubsectionnum. #1.}\fi
\refstepcounter{subsectionnum}}

\newtheorem{theor}{Theorem}[section]
\newtheorem{rem}[theor]{Remark}
\newtheorem{cor}[theor]{Corollary}

\newcommand{\CC}{{\mathbb C}}
\newcommand{\RR}{{\mathbb R}}

\newcommand{\QQ}{{\mathbb Q}}
\newcommand{\FF}{{\mathbb{ F}}}
\newcommand{\ZZ}{{\mathbb Z}}

\newcommand{\G}{{\mathbf G}}

\newcommand{\Ga}{\Gamma}

\newcommand{\ol}{\overline}

\newcommand{\lo}{\longrightarrow}

\newcommand{\GL}{{\rm \bf GL}}

\newcommand{\rk}{\textnormal{rk}\,}

\newcommand{\cO}{{\mathcal O}}

\newcommand{\SU}{{\mathbf{ SU}}}

\newcommand{\SL}{{\mathbf{SL}}}

\newcommand{\OO}{\mathcal{O}}

\newcommand{\Ad}{\textnormal{Ad}\,}

\newcommand{\car}{\textnormal{char}}

\begin{document}


\maketitle



\begin{prelims}


\def\abstractname{Abstract}
\abstract{We show that any lattice in a simple $p$-adic Lie group is not the
fundamental group of a compact K\"ahler manifold, as well as some
variants of this result.}

\vspace{0.10cm}


\vspace{0.10cm}


\vspace{0.35cm}

\languagesection{Fran\c{c}ais}{%

\textbf{Titre. Les r\'eseaux $p$-adiques ne sont pas des groupes k\"ahl\'eriens} \commentskip \textbf{R\'esum\'e.} Dans cette note, nous montrons qu'un r\'eseau d'un groupe de Lie $p$-adique simple n'est pas le groupe fondamental d'une vari\'et\'e k\"ahl\'erienne compacte, ainsi que des variantes de ce r\'esultat.}

\end{prelims}


\newpage

\setcounter{tocdepth}{1} \tableofcontents

\section{Results}

\newsubsection{} \label{K\"ahler groups}
A group is said to be a K\"ahler group if it is isomorphic to the
fundamental group of a connected compact K\"ahler manifold. In particular such a
group is finitely presented. As any finite \'etale cover of a compact
K\"ahler manifold is still a compact K\"ahler manifold, any finite
index subgroup of a K\"ahler group is a K\"ahler group. The most elementary necessary condition
for a finitely presented group to be K\"ahler is that its finite index subgroups have even rank
abelianizations. A classical question, due to Serre and still largely open, is to characterize
K\"ahler groups among finitely presented groups. A standard reference
for K\"ahler groups
is \cite{amoros}.

\newsubsection{} \label{lattices}
In this note we consider the K\"ahler problem for lattices in simple groups
over local fields. Recall that if $G$ is a locally compact topological group, a subgroup
$\Ga \subset G$ is called a {\em lattice} if it is a
discrete subgroup of $G$ with finite covolume (for any $G$-invariant
measure on the locally compact group $G$). 

We work in the following setting.
Let $I$ be a finite set of indices. For each $i \in I$ we fix a
local field $k_i$ and a simple algebraic group $\G_i$ defined and
isotropic over $k_i$. Let $G=
\prod_{i \in I} \G_i(k_i)$. The topology of the local fields $k_i$, $i 
\in I$, makes $G$ a locally compact topological group. We define $\rk G:= \sum_{i \in I} \rk_{k_{i}} \G_i$.

We consider $\Ga \subset G$ an {\em irreducible} lattice: 
there does not exist a disjoint decomposition $I = I_1 \coprod I_2$ into two
non-empty subsets such that, for $j=1,2$, the subgroup $\Gamma_j:= \Gamma
\cap G_{I_{j}}$ of $G_{I_{j}}:= \prod_{i \in I_{j}} \G_i(k_i)$ is a
lattice in~$G_{I_{j}}$. 

The reference for a detailed study of such lattices is \cite{mar}. In
Section \ref{reminder} we recall a few results for the convenience of the reader.

\newsubsection{}
Most of the lattices $\Ga$ as in Section \ref{lattices} are finitely presented (see
Section \ref{finite presentation}). The question whether $\Ga$ is K\"ahler or not has been studied by Simpson 
using his non-abelian Hodge theory when at least one of the $k_i$'s is
archimedean. He shows that if
$\Gamma$ is K\"ahler then necessarily for any $i \in I$ such
that $k_i$ is archimedean the group $\G_i$ has to be of Hodge
type (i.e. admits a Cartan involution which is an inner automorphism),
see \cite[Corollary 5.3 and 5.4]{Simpson}. For
example $\SL(n, \ZZ)$ is not a K\"ahler group
as $\SL(n, \RR)$ is not a group of Hodge type. In this note we prove:

\begin{theor} \label{main theorem}
Let $I$ be a finite set of indices and $G$ be a group of the form
$\prod_{j \in I} \G_j(k_j)$, where $\G_j$
is a simple algebraic group defined and isotropic over a local field $k_j$. Let
$\Gamma \subset G$ be an irreducible lattice. 

Suppose there exists an $i \in I$ such that $k_i$ is
non-archimedean. If $\rk G >1$ and $\car (k_i)=0$, or if $\rk G=1$
(i.e. $G= \G(k)$ with $\G$ a simple isotropic algebraic group of rank $1$
over a local field $k$)
then $\Ga$ is not a K\"ahler group.
\end{theor}

\begin{rem} \label{rem1} {\rm
Notice that the case $\rk G=1$ is essentially folkloric. As we did not find a
reference in this generality let us give the proof in this case. 

If $\Ga$ is not cocompact in $G$ (this is possible only if $k$ has
positive characteristic) then $\Ga$ is not finitely generated by
\cite[Corollary 7.3]{L}, hence not K\"ahler.

Hence we can assume that $\Ga$ is cocompact.
In that case it follows from \cite[Theorem 6.1 and 7.1]{L} that
$\Ga$ admits a finite index subgroup $\Ga'$ which is a (non-trivial)
free group. But a non-trivial free group is never K\"ahler, as it always admits a
finite index subgroup with odd Betti number (see \cite[Example 1.19
p.7]{amoros}). Hence $\Ga'$, thus also $\Ga$, is not K\"ahler.}
\end{rem}

On the other hand, to the best of our knowledge
not a single case of Theorem \ref{main theorem} in the case where $\rk G >1$
and all the $k_i$, $i \in I$, are non-archimedean fields of
characteristic zero was previously
known. The proof in this case is a corollary of Margulis'
superrigidity theorem and the recent integrality result of Esnault and
Groechenig (\cite[Theorem 1.3]{EG}, whose proof was greatly simplified in \cite{EG2}).

\newsubsection{} Let us mention some examples of Theorem \ref{main theorem}:

-- Let $p$ be a prime number, $I= \{1\}$, $k_1= \QQ_p$ , $\G =
\SL(n)$. A lattice in $\SL(n, \QQ_p)$, $n \geq 2$,
is not a K\"ahler group. This is new for $n \geq 3$.

-- $I=\{1;2\}$, $k_1= \RR$ and $\G_1 = \SU(r,s)$ for some
$r\geq s>0$, $k_2= \QQ_p$ and $\G_2 = \SL(r+s)$. Then any irreducible
lattice in $SU(r,s) \times \SL(r+s, \QQ_p)$ is not K\"ahler. In
Section \ref{reminder} we recall how to construct such lattices (they are
$S$-arithmetic). The analogous result that any irreducible lattice in
$\SL(n, \RR) \times \SL(n, \QQ_p)$ (for example $\SL(n, \ZZ[1/p])$) is
not a K\"ahler group already followed from Simpson's theorem.

\newsubsection{} I don't know anything about the case not
covered by Theorem \ref{main theorem}: can a (finitely presented) irreducible lattice in
$G=\prod_{i \in I} \G_i(k_i)$ with $\rk G>1$ and all $k_i$ of (necessarily
the same, see Theorem \ref{arithmeticity}) {\em positive characteristic}, be
a K\"ahler group? This question already appeared in \cite[Remark 0.2 (5)]{BKT}.

\section{Reminder on lattices} \label{reminder}

\newsubsection{} \label{S-arithmetic}
Examples of pairs $(G, \Ga)$ as in Section \ref{lattices} are provided by
{\em $S$-arithmetic groups}: let $K$ be a global field (i.e a finite extension of $\QQ$ or
$\FF_q(t)$, where $\FF_q$ denotes the finite field with $q$ elements),
$S$ a non-empty set of places of $K$, $S_\infty$ the set of
archimedean places of $K$ (or the 
empty set if $K$ has positive characteristic), $\cO^{S \cup
  S_{\infty}}$ the ring of elements of $K$ which are integral at all
places not belonging to $S \cup S_{\infty}$ and $\G$ an absolutely
simple $K$-algebraic group, anisotropic at all archimedean places not
belonging to $S$. A subgroup $\Lambda \subset \G(K)$ is said
$S$-arithmetic (or $S \cup S_\infty$-arithmetic) if it is
commensurable with $\G(\cO^{S \cup
  S_{\infty}})$ (this last notation depends on the choice of an
affine group scheme flat of finite type over $\cO^{S \cup
  S_{\infty}}$, with generic fiber
$\G$; but the commensurability class of the group $\G(\cO^{S \cup
  S_{\infty}})$ is independent of that choice). 

If $S$ is finite the image $\Ga$ in $\prod_{v \in S} \G(K_v)$ of an $S$-arithmetic
group $\Lambda$ by the diagonal map is an irreducible lattice (see \cite{B}
in the number field case and \cite{H} in the function field case). 
In the situation of Section \ref{lattices}, a (necessarily irreducible) lattice $\Ga \subset G$ is called
$S$-arithmetic if there exist $K$, $\G$, $S$ as above, a bijection
$i: S \lo I$, isomorphisms $K_v \lo k_{i(v)}$ and, via these
isomorphisms, $k_i$-isomorphisms $\varphi_i: \G \lo \G_i$ such that
$\Ga$ is commensurable with the image via $\prod_{i \in I} \varphi_i$
of an $S$-arithmetic subgroup of $\G(K)$. 

\newsubsection{}
Margulis' and Venkataramana's
arithmeticity theorem states that as soon as $\rk G$ is at least $2$
then every irreducible lattice in $G$ is of this type:

\begin{theor}[Margulis, Venkataramana] \label{arithmeticity}
In the situation of Section \ref{lattices}, suppose that $\Ga \subset G$ is an irreducible
lattice and that $\rk G \geq 2$. Suppose moreover for simplicity that $\G_i$, $i \in I$,
is absolutely simple. Then:
\begin{itemize}
\item[\rm (a)] All the fields $k_i$ have the same characteristic.
\item[\rm (b)] The group $\Ga$ is $S$-arithmetic.
\end{itemize}

\end{theor}

\begin{rem} {\rm
Margulis \cite{Mar1} proved Theorem \ref{arithmeticity} when $\car (k_i) =0$ for all $i
\in I$. Venkatarama \cite{V} had to overcome many technical difficulties in
positive characteristics to extend Margulis' strategy to the general case.
}
\end{rem}

On the other hand, if $\rk G=1$ (hence $I=\{1\}$) and $k = k_1$ is
non-archimedean, there exist non-arithmetic lattices 
in $G$, see \cite[Theorem A]{L}.

\newsubsection{} \label{finite presentation}

With the notations of Section \ref{S-arithmetic}, an $S$-arithmetic lattice
$\Ga$ is always finitely presented except if $K$ is a function field,
and $\rk_K \G= \rk G = |S|=1$ (in which case $\Ga$ is not even
finitely generated) or $\rk_K \G>0$ and $\rk G=2$ (in which case $\Ga$
is finitely generated but not finitely presented). In the number field
case see the result of Raghunathan \cite{rag} in the classical
arithmetic case and of Borel-Serre \cite{BS} in the general $S$-arithmetic
case; in the function field case see the work of Behr, e.g. \cite{Behr}. For example the
lattice $\SL_2(\FF_q[t])$ of $\SL_2(\FF_q((1/t)))$ is not finitely
generated, while the lattice $\SL_3(\FF_q[t])$ of
$\SL_3(\FF_q((1/t)))$ is finitely generated but not finitely presented.

\section{Proof of Theorem \ref{main theorem}}

Thanks to Remark \ref{rem1} we can assume that $\rk G >1$.
In this case the main tools for proving Theorem \ref{main theorem} are the
recent result of Esnault and Groechenig and
Margulis' superrigidity theorem.

\newsubsection{}

Recall that a linear representation $\rho: \Ga \lo \GL(n, k)$ of a
group $\Gamma$ over a field $k$ is cohomologically rigid if $H^1(\Ga,
\Ad \rho) =0$. A representation $\rho: \Ga \lo \GL(n, \CC)$ is said to
be integral if it factorizes through $\rho: \Ga \lo \GL(n, K)$, $K
\hookrightarrow \CC$ a number field, and moreover stabilizes an
$\OO_K$-lattice in $\CC^n$ (equivalently, see \cite[Corollary 2.3 and 2.5]{Bass}: for any embedding $v: K
\hookrightarrow k$ of $K$ in a non-archimedean
local field $k$ the composed representation $\rho_v: \Ga \lo \GL(n, K)
\hookrightarrow \GL(n, k)$
has bounded image in $\GL(n, k)$). A group will be said {\em complex projective} if is isomorphic to the fundamental
group of a connected smooth complex projective variety. This is a
special case of a K\"ahler group (the question whether or not any
K\"ahler group is complex projective is open).

In \cite[Theorem 1.1]{EG2} Esnault and Groechenig prove that if $\Gamma$ is a complex projective group then 
any irreducible cohomologically rigid representation $\rho: \Ga \lo
\GL(n, \CC)$ is integral. This was conjectured by Simpson.

\newsubsection{}

A corollary of \cite[Theorem 1.1]{EG2} is the following:

\begin{cor} \label{lem1}
Let $\Ga$ be a complex projective group. Let $k$ be a
non-archimedean local field of characteristic zero and let $\rho: \pi_1(X) \lo \GL(n, k)$ be
an absolutely irreducible cohomologically rigid representation. Then
$\rho$ has bounded image in $\GL(n, k)$.
\end{cor}

\begin{proof}
Let $\ol{k}$ be an algebraic closure of $k$.
As $\rho$ is absolutely irreducible and cohomologically rigid there 
exists $g \in \GL(n, \ol{k})$ and a number field $K \subset \ol{k}$ such that $\rho^g(\Gamma):=
g\cdot \rho \cdot g^{-1}(\Ga) \subset \GL(n, \ol{k})$ lies in
$\GL(n, K)$. 

Let $k'$ be the finite extension of $k$ generated by 
$K$ and the matrix coefficients of $g$ and $g^{-1}$. This is still a non-archimedean
local field of characteristic zero, and both $\rho(\Gamma)$ and
$\rho^g(\Gamma)$ are subgroups of $\GL(n, k')$. As $\rho: \Ga \lo \GL(n, k) \subset \GL(n, k')$
has bounded image in $\GL(n, k)$ if and only if $\rho^g: \Ga \lo \GL(n,
k')$ has bounded image in $\GL(n, k')$, we can
assume, replacing $\rho$ by $\rho^g$ and $k$ by $k'$ if necessary,
that $\rho(\Gamma)$ is contained in $\GL(n, K)$ with $K \subset k$ a number field.

Let $\sigma: K \hookrightarrow \CC$ be an infinite place of $K$ and
consider $\rho^\sigma: \Ga \stackrel{\rho}{\lo} \GL(n, K)
\stackrel{\sigma}{\hookrightarrow} \GL(n,\CC)$ the associated
representation. As $\rho$ is absolutely irreducible, the
representation $\rho^\sigma$ is irreducible. As 
$$H^1(\Ga, \Ad \circ \rho^\sigma) = H^1(\Ga, \Ad \circ \rho)\otimes_{K, \sigma} \CC =0$$
the representation $\rho^\sigma$ is cohomologically rigid.

It follows from \cite[Theorem 1.3]{EG} that
$\rho^\sigma$ is integral. In particular, considering the embedding $K
\subset k$, it follows that the representation $\rho: \Ga \lo \GL(n, k)$ 
has bounded image in $\GL(n, k)$.
\hfill $\Box$
\end{proof}

\newsubsection{}
Notice that we can upgrade Corollary \ref{lem1} to the K\"ahler world if we restrict
ourselves to faithful representations:

\begin{cor} \label{lem2}
The conclusion of Corollary \ref{lem1} also holds for $\Ga$ a K\"ahler group 
and $\rho: \pi_1(X) \lo \GL(n, k)$ a {\em faithful} representation.
\end{cor}

\begin{proof}
As the representation $\rho$ is faithful, the group $\Ga$ is a linear
group in characteristic zero. It then follows that the
K\"ahler group $\Ga$ is a complex projective group
(see \cite[Theorem 0.2]{CCE14} which proves that a finite index subgroup of
$\Ga$ is complex projective, and its refinement \cite[Corollary 1.3]{C} which proves
that $\Ga$ itself is complex projective). The result now
follows from Corollary \ref{lem1}.
\hfill $\Box$
\end{proof}

\newsubsection{}
Let us apply Corollary \ref{lem1} to the case of Theorem \ref{main theorem} where $\rk G>1$. Renaming the indices of $I$ if
necessary, we can assume that $I = \{1, \cdots, r\}$ and $k_1$ is
non-archimedean of characteristic zero. Let us choose an absolutely irreducible $k_1$-representation $\rho_{\G_{1}}:
\G_1 \lo \GL(V)$. Let $$\rho: \Ga \lo G \stackrel{p_{1}}{\lo} \G_1(k_1) \lo
\GL(V)$$ be the representation of $\Ga$ deduced from $\rho_{\G_{1}}$
(where $p_1: G \lo  \G_1(k_1) $ denotes the projection of $G$ onto its
first factor). As $p_1(\Ga)$ is Zariski-dense in $\G_1$ it follows that
$\rho$ is absolutely irreducible.

As $\rk G >1$, Margulis' superrigidity
theorem applies to the lattice $\Ga$ of $G$:  it implies in particular
that $ H^1(\Ga, \Ad \circ \rho) =0$ (see
\cite[Theorem (3)\,(iii) p.\,3]{mar}). Hence the representation $\rho: \Ga
\lo \GL(V)$ is cohomologically rigid.

Suppose by contradiction that $\Ga$ is a K\"ahler group. By
Theorem \ref{arithmeticity}\,(a) and the assumption that $k_1$ has
characteristic zero it follows that $\Ga$ is linear in characteristic
zero. As in the proof of Corollary \ref{lem2} we deduce that $\Ga$ is a
complex projective group. It then follows from Corollary \ref{lem1} that $\rho$
has bounded image in $\GL(V)$, hence that $p_1(\Ga)$ is relatively compact in
$\G(k_1)$. This contradicts the fact that $\Ga$ is a lattice in
$G= \G(k_1) \times \prod_{j \in I \setminus\{1\}} \G(k_j)$.
\hfill $\Box$

\providecommand{\bysame}{\leavevmode\hbox to3em{\hrulefill}\thinspace}
%
%

\bibliographystyle{amsalpha}
\bibliographymark{References}
\def\cprime{$'$}

\end{document}